\documentclass{article}[11pt]
\usepackage{amsthm,amssymb,amsmath}

\newcommand{\cross}{\cdot}

\newcommand{\ga}{{\frak a}}

\newcommand{\cO}{{\cal O}}

\newcommand{\Z}{{\bf Z}}
\newcommand{\F}{{\bf F}}

\newcommand{\Aut}{{\cal A}{\it ut}}

\newcommand{\Tr}{{\cal T}{\it r}}
\newcommand{\Ctf}{{\cal C}{\it tf}}
\newcommand{\Tw}{{\bf Tw}}
\newcommand{\Gal}{{\bf Gal}}

\newtheorem{theo}[subsection]{Theorem}
\newtheorem{lem}[subsection]{Lemma}
\newtheorem{prop}[subsection]{Proposition}

\newtheorem{exam}[subsection]{Example}

\usepackage[all]{xy}

\begin{document}
\author{A. Davydov}
\title{Twisted automorphisms of group algebras}
\maketitle
\begin{center}
Department of Mathematics, Division of Information and Communication Sciences, Macquarie University, Sydney, NSW 2109,
Australia
\end{center}
\begin{center}
davydov@math.mq.edu.au
\end{center}
\begin{abstract}
We continue the study of twisted automorphisms of Hopf algebras started in \cite{da2}. In this paper we concentrate on
the group algebra case. We describe the group of twisted automorphisms of the group algebra of a group of order coprime
to 6. The description turns out to be very similar to the one for the universal enveloping algebra given in \cite{da2}.
\end{abstract}
\section{Introduction}
In \cite{da1}, using the language of Galois algebras, monoidal (auto-)equivalences of categories of representations of
finite groups were described in terms of some group-theoretic data. Composition of monoidal equivalences corresponding
to tensor product of Galois algebras turns out to have a quite complicated form in terms of that data. In particular, the
group structure on isomorphism classes of monoidal auto-equivalences (bi-Galois algebras) is not very easy to deal
with.

Here we describe this structure using a different presentation (as {\em twisted automorphisms}) for monoidal
(auto-)equivalences, which was developed in \cite{da2}. Since any Galois algebra over a finite group has a normal basis
(\cite{mo}) the results of \cite{da1,da2} imply that any bi-Galois algebra corresponds to a twisted automorphism. In
\cite{da2} general structure of the (Cat-)group of twisted automorphisms of a Hopf algebra was examined. As an example,
the case of the universal enveloping algebra of a Lie algebra (over the ring of formal power series) was treated. In
that situation any twisted automorphism is a bialgebra automorphism together with an invariant twist (the so-called
{\em separated} case). The gauge classes of invariant twists form an abelian group isomorphic to the invariant elements
of the exterior square of the Lie algebra. The group of gauge classes of twisted automorphisms is a crossed product of
the group of automorphisms of the Lie algebra and the group of invariant twists.

For the group algebra case the situation is more complicated yet resembles closely the case of a universal enveloping
algebra. For simplicity we restrict ourselves to groups of odd order. We still have the separation property but
separation into a composition of a group automorphism and an invariant twist is not unique: there are invariant twists
({\em symmetric} twists), which are gauge isomorphic to group automorphisms ({\em class-preserving} automorphisms).
However we can make it unique by using only anti-symmetric twists. The gauge classes of {\em anti-symmetric} twists form
an abelian group $Aut_{\Tw}^{anti}(k[G])$ with the operation closely related to the one on twists of a universal
enveloping algebra. If the order of the group is coprime to 6, the group of anti-symmetric twists is a subgroup of the
group of all invariant twists $Aut_{\Tw}^1(k[G])$. Moreover, the group of invariant twists is the direct product of the
group of anti-symmetric twists and the group of symmetric twists (the group of class-preserving outer automorphisms):
$$Aut_{\Tw}^1(k[G])\simeq Out_{cl}(G)\times Aut_{\Tw}^{anti}(k[G]).$$

We can summarise the main results of the paper in the form of a commutative diagram with exact rows and columns:
$$\xymatrix{ Out_{cl}(G) \ar[r] \ar[d]  & Out(G) \ar[r] \ar[d] &  Out(G)/ Out_{cl}(G) \ar[d]
\\ Aut_{\Tw}^1(k[G]) \ar[r] \ar[d] &  Aut_{\Tw}(k[G]) \ar[r] \ar[d]  & Out(G)/ Out_{cl}(G) \\
Aut_{\Tw}^{anti}(k[G]) \ar[r] & Aut_{\Tw}^{anti}(k[G]) }$$ Note that in general the subgroup $Out(G)$ (which
can be identified with the group of {\em symmetric} twisted homomorphisms) is not normal in $ Aut_{\Tw}(k[G])$. The
coset $Aut_{\Tw}(k[G])/Out(G)$ is naturally identified with the set of {\em triangular structures} on the Hopf algebra
$k[G]$ with trivial {\em Drinfeld invariant}. If the order of the group is coprime to 6, this set has a group structure
(the group of anti-symmetric twists).

Throughout the paper let $k$ be a ground field, which is supposed to be algebraically closed of characteristic zero.

\section*{Acknowledgment}
The paper was started during the author's visit to the Max-Planck Institut f\" ur Mathematik (Bonn). The author would
like to thank this institution for hospitality and excellent working conditions. The work on the paper was supported by
Australian Research Council grant DP00663514. Special thanks are due to R. Street for invaluable support during the work on the paper.

\section{Twists on group algebras}
Here we revise Movshev's classification of twists on the group algebra $k[G]$ of a finite group in terms of Galois
algebras \cite{mo}. An important ingredient of that classification is a group with a non-degenerate 2-cocycle. Recall
that a group 2-cocycle $\alpha\in Z^2 (S,k^\cross)$ is {\em non-degenerate} if, for any $s\in S$, the homomorphism
$$C_G(s)\to k^\cross,\quad x\mapsto \alpha(x,s)\alpha(s,x)^{-1}$$ is nontrivial. Note that a cocycle cohomologous to a
non-degenerate cocycle is also non-degenerate. There is a correspondence between groups with non-degenerate 2-cocycles
and the so-called groups of central type. A group is of {\em central type} if it has an irreducible representation of
dimension equal to the index of its centre (the largest possible). The quotient by the centre of a group of central type
is a group with non-degenerate 2-cocycle (thus the alternative name {\em central type factor group}). Conversely a
group with non-degenerate 2-cocycle has a central extension, which is a group of central type. Using the classification
of finite simple groups, Howlett and Isaacs proved that groups of central type are solvable (the conjecture of Iwahori
and Matsumoto) \cite{hi}.
\begin{theo}\label{gal}
The groupoid $\Gal(k[G])$ of Galois $G$-algebras is equivalent to the groupoid $\Ctf(G)$ of pairs $(S,\alpha)$, where
$S$ is a subgroup of $G$ and $\alpha\in Z^2(S,k^\cross)$ is a non-degenerate 2-cocycle. A morphism $(S,\alpha)\to
(T,\beta)$ is a pair $(g,c)$, where $g\in G$ and $c:S\to k^\cross$ such that $gSg^{-1} = T$ and $$\alpha(x,y)c(xy) =
\beta(gxg^{-1},gyg^{-1})c(gxg^{-1})c(gyg^{-1}),\quad x,y\in S.$$
\end{theo}
For the proof of this slightly modified version of Movshev's result see \cite{da1}. A Galois $G$-algebra corresponding
to the pair $(S,\alpha)$ can be realized explicitly as the algebra of functions: $$R(G,S,\alpha) = \{r:G\to
k[S,\alpha]:\quad r(sg) = s(r (g))\quad\forall s\in S, g\in G\}$$ with the $G$-action given by $(f\phi )(g) = \phi
(gf)$.

It was proved in \cite{mo} (see also \cite{da1}) that all Galois $G$ algebras possess normal bases thus (according to section 2.2. of
\cite{da2}) establishing the equivalence of the groupoids $\Aut_\Tw(k[G])\to \Aut_\Gal(k[G])$. Although a
quasi-inverse to this equivalence is very hard to construct in general, in some cases twists corresponding to Galois
algebras can be written explicitly. Note that, for an abelian $A$, the alternation map $$\alpha\mapsto Alt(\alpha),\quad
Alt(\alpha)(s,t) = \alpha(s,t)\alpha(t,s)^{-1}$$ identifies the cohomology group $H^2(A,k^\cross)$ with the group
$Hom(\Lambda^2 A,k^\cross)$ of {\em alternative} bi-multiplicative forms $$\beta:A\times A\to k^\cross,\quad \beta(s,s) = 1\ \forall
s\in A.$$ Moreover a bi-multiplicative form (not necessarily alternative) is always a 2-cocycle. It is not hard to see
that any alternative form on an abelian group is the alternation of some bi-multiplicative form. Thus any 2-cocycle on
an abelian $A$ is cohomologous to some bi-multiplicative form. It follows from the definition that a bi-multiplicative
form is non-degenerate as a 2-cocycle if and only if it is non-degenerate in the ordinary sense (establishes an
isomorphism $A\to \hat A$). For a non-degenerate bi-multiplicative form $\beta$ on abelian $A$ denote by $b$ the {\em
adjoint} bi-multiplicative form on $\hat A$. Note that $b$ is also non-degenerate. Define $$F_{(A,b)} =
\sum_{\psi,\chi\in\hat A}b(\psi,\chi)p_\psi\otimes p_\chi,$$ where $p_\chi = \frac{1}{|A|}\sum_{s\in S}\chi(s)^{-1}s$
is the minimal idempotent in $k[A]$ corresponding to the character $\chi$. In the next lemma we will verify
that $F_{(A,b)} $ (as an element of $k[G]^{\otimes 2}$) is the twist corresponding to the Galois $G$-algebra
$R(G,A,\beta)$.
\begin{lem}
Let $F=F_{(A,b)}$ be a twist corresponding to an abelian subgroup $A\subset G$ and a bi-multiplicative form
$\beta$. Then the function algebra $(k(G),*_F)$ with the $F$-twisted multiplication is isomorphic to the Galois
$G$-algebra $R(G,A,\beta)$.
\end{lem}
\begin{proof}
First note that $(k(G),*_F)$ coincides with the induction algebra $ind^G_S(B)$ where $B$ is the function algebra
$(k(A),*_F)$ with the $F$-twisted multiplication. So all we need to check is that $(k(A),*_F)$ is isomorphic to the skew
group algebra $k[A,\beta]$. Note that $k[A,\beta]$ is (canonically) isomorphic to $k[\hat A,b]$. Now an isomorphism
$k[\hat A,b]\to (k(A),*_F)$ can be established explicitly by assigning $e_\chi\mapsto l_\chi=\sum_{a\in A}\chi(a)p_a$.
Indeed, since $p_\chi(l_\psi) = \delta_{\chi,\psi}l_\psi$ and $l_\chi i_\psi = l_{\chi\psi}$, the twisted product has the
form $$l_\chi*_F l_\psi = \sum{\chi',\psi'\in\hat A}b(\chi',\psi')p_{\chi'}(l_\chi)p_{\psi'}(l_\psi) =
b(\chi,\psi)l_\chi l_\psi = b(\chi,\psi)l_{\chi\psi}.$$
\end{proof}

Automorphisms of Galois $G$-algebras were studied in \cite{da1}. In view of theorem \ref{gal}, they are
automorphisms of objects of the groupoid $\Ctf(G)$ of central type factor subgroups of $G$.
\begin{theo}\label{age}
The group of $G$-automorphisms $Aut_G(R)$ of the Galois $G$-algebra $R=R(G,S,\alpha)$ fits into an exact sequence:
$$1 \to  \hat S \to Aut_G(R) \to St_{N_G(S)/S}(A).$$
\end{theo}
Here $\hat S = Hom(S,k^\cross)$ is the group of characters of $S$, $N_G(S)$ is the normaliser of $S$ in $G$, and $A\in
H^2(S,k^\cross)$ is the class of $\alpha$. In particular, the automorphism group $Aut_G(R)$ has the same order as $G$
only if $S$ is normal abelian and the class $A$ is $G$-invariant.

Now following \cite{da,da1} we will prove that (up to gauge equivalence) twists for twisted isomorphisms are
supported by normal abelian subgroups. Very similar results were proved in \cite{eg}.
\begin{theo}\label{absup}
Any twisted isomorphism $(f,F)$ of group algebras $k[G_1]\to k[G_2]$ is gauge isomorphic to a twisted isomorphism of
the form $(f',F_{(S,\beta)})$ for some normal abelian $S\subset G$ and a non-degenerate bi-multiplicative form $\beta$
with $G$-invariant alternation.
\end{theo}
\begin{proof}
For a twisted automorphism $(f,F)$ of $k[G]$ the $f$-twist $F$ must correspond to a Galois $G$-algebra $R$ with the
automorphism group isomorphic to $G$. Thus $R$ is isomorphic (as a Galois $G$-algebra) to $R(G,S,\alpha)$ for a normal
abelian $S$ with a non-degenerate 2-cocycle with $G$-invariant cohomology class. Finding a bilinear form $b$ with the
same class, we can replace $R(G,S,\alpha)$ with (the isomorphic Galois $G$-algebra) $R(G,S,b)$. Finally the Galois
$G$-algebra $R(G,S,b)$ corresponds to the twist $F_{(S,\beta)}$. Thus the twist $F$ is gauge isomorphic to the twist
$F_{(S,\beta)}$: $$F_{(S,\beta)}\Delta(a) = (a\otimes a)F$$ for some invertible $a\in k[G]$, which can be viewed as a
gauge transformation $a:(f,F)\to (f',F_{(S,\beta)})$, where $f'(x) = af(x)a^{-1}$.
\end{proof}

\section{Separation of twisted isomorphisms}

The fact that twists of twisted isomorphisms are supported by abelian normal subgroups allows us to prove separation
for twisted automorphisms of the group algebra of a group of odd order.
\begin{prop}\label{separ}
For a group $G$ of odd order any twisted automorphism $(f,F)$ of the group algebra $k[G]$ is gauge isomorphic to a
unique twisted automorphism of the form $(f',F_{(S,\beta)})$, where $f'$ is an automorphism of $k[G]$ induced by a
group automorphism of $G$, where $S$ is normal abelian, and where $\beta$ is a non-degenerate alternative $G$-invariant form.
\end{prop}
\begin{proof}
For odd order $S$ the alternation map is bijective on $Hom(\Lambda^2S,k^\cross)$. Thus we can strengthen the statement
of the theorem \ref{absup} assuming that $\beta$ is a non-degenerate alternative $G$-invariant form. This assumption
makes the twist $F_{(S,\beta)}$ $G$-invariant. Hence the twisted automorphism $(f',F_{(S,\beta)})$ is separated and
$f'$ is a bialgebra automorphism of $k[G]$, which has to be induced by a group automorphism of $G$.
\end{proof}

Below we show that group algebras of groups of even order can have non-separated twisted isomorphisms.

Let $\beta:\hat A\times\hat A\to k^\cross$ be a non-degenerate bi-multiplicative form on the dual group of an abelian
normal subgroup $A\subset G$ with $G$-invariant alternation $$b(g(\chi)g(\xi)) = b(\chi,\xi),\quad b(\chi,\xi) =
\beta(\chi,\xi)\beta(\xi,\chi)^{-1}.$$ This means that for any $g\in G$ the form $\beta_g(\chi,\xi) =
\beta(g(\chi),g(\xi))\beta(\chi,\xi)^{-1}$ is symmetric. Define $$\tilde G = \{(g,c)| g\in G, c:\hat A\to k^\cross,\
\beta_g(\chi,\xi) = c(\chi)c(\xi)c(\chi\xi)^{-1},\ \forall \chi,\xi\in\hat A\}.$$ It is a group with the product
$$(g_1,c_1)(g_2,c_2) = (g_1g_2,(c_1)^{g_2}c_2),$$ where $c^g(\chi) = c(g^{-1}(\chi))$. Define a homomorphism
$f:k[\tilde G]\to k[G]$ by $$f(g,c) = g\sum_{\chi\in \hat A}c(\chi)p_\chi$$ and the twist $$F = \sum_{\chi,\xi\in \hat
A}\beta(\chi,\xi)p_\chi\otimes p_\xi,$$ where $ p_\chi = \frac{1}{|A|}\sum_{a\in A}\chi(a)^{-1}a$. By the definition
$$\Delta'(f(g,c))F = F(f(g,c)\otimes f(g,c))$$ so that $(f,F)$ is a twisted homomorphism. The subgroup $$K =
\{(a,c_a^{-1}),\ a \in A\}\subset \tilde G$$ is mapped into 1, which means that $(f,F)$ induces a twisted homomorphism
$(f,F):k[\overline G]\to k[G]$, where $\overline G = \tilde G/K$. Note that pairs $(a,1)$, where $1(\chi) = \chi(1)$,
form a normal subgroup in $\tilde G$($\overline G$) isomorphic to $A$. Moreover, the quotient group $\overline G/A$ is
isomorphic to the quotient group $Q = G/A$ and the $Q$-action on $A$, coming from $\overline G$, coincides with the one
coming from $G$. The class of the extension $A\to \overline G\to Q$ in $H^2(Q,A)$ was calculated in \cite{da1}. It is 
equal to the class of $A\to G\to Q$ shifted by the image of $b\in Hom(\lambda^2(\hat A),k^\cross) = H^2(\hat
A,k^\cross)$ under the map $$D:H^0(Q,H^2(\hat A,k^\cross))\to H^2(Q,H^1(\hat A,k^\cross)) = H^2(Q,A),$$ which is the
differential of the second term of the Hochschild-Serre spectral sequence associated with the split extension of $A$
by $Q$. The map $D$ is trivial in the absence of 2-torsion. Thus for odd order $G$ the twisted homomorphism $(f,F)$ is
a twisted automorphism and, in particular, is separable. In the case when $\overline G$ is not isomorphic to $G$ the
twisted homomorphism $(f,F)$ can not be separable. Of course non-triviality of $D(b)$ is not enough to guarantee that $G$ is not isomorphic to $\overline G$ as abstract groups, but in the
next example (taken from \cite{da1}, see also \cite{eg}) it can be checked directly.

\begin{exam}Affine symplectic and metaplectic groups.
\end{exam}
Let $A$ be an elementary abelian 2-group with a symplectic form $(\ ,\ ):A\otimes A\to \F_2$. Let $Q$ be the group of
automorphisms of $A$ preserving $(\ ,\ )$, i.e. the symplectic group $Sp(n,2)$, where $n$ is the rank of $A$. Let $G$
be the semi-direct product of $A$ and $Q$, i.e. the affine symplectic group $ASp(n,2)$. Let $m:A\otimes A\to \F_2$ be a
bi-linear form such that $m(x,y) - m(y,x) = (x,y)$. Define a non-degenerate bi-multiplicative form $\beta:A\otimes A\to
k^\cross$ by $$\beta(x,y) = (-1)^{m(x,y)}.$$ By the definition the alternation of $\beta$ is $G$-invariant. The group
$\overline G$ corresponding to the abelian normal subgroup $A$ with the form $\beta$ coincides with the quotient of the
group of pairs
\begin{multline*}
\{(g,q)| g\in Sp(n,2), q:A\to \F_2, \\
m(g(x),g(y)) - m(x,y) = q(x) = q(y) - q(x+y),\ \forall x,y\in
A\},
\end{multline*}
which is called the {\em metaplectic group} $Mp(n,2)$ \cite{we}. It is known that for $n>2$ the metaplectic
group $Mp(n,2)$ is not isomorphic to the affine symplectic group $ASp(n,2)$. Thus the twisted homomorphism
$$(f,F):k[ASp(n,2)]\to k[Mp(n,2)].$$ is not separable.

\section{Class-preserving automorphisms}\label{clpr}
Here we describe invariant twists, which stabilise the unital inclusion $k\to k[G]$ (as in the section 4.3 of \cite{da2}) and
link them with class-preserving automorphisms of $G$. Recall that an automorphism $f:G\to G$ is {\em class-preserving}
if $f$ preserves conjugacy classes of $G$: for any $g\in G$, $$f(g) = xgx^{-1},\ \mbox{for some}\ x\in G.$$ Clearly,
class-preserving automorphisms are closed under composition, forming a normal subgroup $Aut_{cl}(G)$ of the
automorphism group $Aut(G)$. Inner automorphisms are obviously class-preserving. Moreover, the homomorphism
\begin{equation}\label{crmcl}
Aut_{cl}(G) \stackrel{\partial}{\leftarrow} G
\end{equation}
sending $g\in G$ to the inner automorphism $(\ )^g$ defines a crossed module of groups $\Aut_{cl}(G)$. In particular,
$\pi_0(\Aut_{cl}(G)) = Out_{cl}(G)$ is the group of class-preserving automorphisms modulo inner.

The study of class-preserving automorphisms was initiated in \cite{bu}, where the first examples of groups with
nontrivial $Out_{cl}(G)$ were constructed. Since then many more examples were produced and some triviality results were
proved. In particular, using the classification of finite simple groups, Sah proved solvability of $Out_{cl}(G)$
\cite{sa} and Feit and Seitz verified triviality of $Out_{cl}(G)$ for simple $G$ \cite{fs}. Recently, class-preserving
automorphisms were used to produce a counter-example to the isomorphism problem of integer group rings \cite{he}.

Here we characterize class-preserving automorphisms as invariant twists on the group algebra stabilizing the unital
inclusion $k\to k[G]$.
\begin{prop}
The homomorphisms $G\to k[G]^\cross,\ Aut_{cl}(G)\to Aut_{bialg}(k[G])$ induce an isomorphism of crossed complexes of
groups $\Aut_{cl}(G)\to\Aut^{inn}_{bialg}(k[G])$.
\end{prop}
\begin{proof}
Since $k$ is algebraically closed (of characteristic zero), $k[G]$ is a sum of matrix algebras and the restriction map
$Out_{bialg}(k[G])\to Aut(Z(k[G]))$ is an isomorphism. The centre $Z(k[G])$ is spanned by class sums (sums over
conjugacy classes of $G$). Thus the kernel of $Aut_{bialg}(k[G])\to Out_{bialg}(k[G])$ coincides with the kernel of
$Aut(G)\to Aut(Z(k[G]))$ and is $Aut_{cl}(G)$.
\end{proof}

We will say that a class preserving automorphism $\phi$ is {\em supported} by a normal subgroup $N\subset G$ if there is an
element $x\in k[N]$ such that $\phi(g) = xgx^{-1}$. Below we give a cohomological description of class-preserving
automorphisms supported by abelian normal subgroups.
\begin{prop}
Let $A$ be an abelian normal subgroup of $G$. Then the group $$\{ x\in k[A]|\ (x\otimes x)\Delta(x)^{-1}\in
(k[A]^{\otimes 2})^G\}$$ is isomorphic to the group of 1-cocycles $\psi:G\to A$ (with respect to the natural $G$-action
on $A$) such that
\begin{equation}\label{absupcl}
\forall \chi\in \hat A,\quad \chi(\psi(s)) = 1,\quad \forall s\in St_G(\chi).
\end{equation}
Here $St_G(\chi) = \{g\in G|\ \chi(g(a)) = \chi(a) \ \forall a\in A\}$ is the stabilizer in $G$ of $\chi\in \hat A$.
Under that isomorphism, central $x$ correspond to coboundaries.
\end{prop}
\begin{proof}
For an element $x\in k[A]$ such that $(x\otimes x)\Delta(x)^{-1}$ is $G$-invariant, define a 1-cocycle $\psi(g) = [x,g]
= xgx^{-1}g^{-1}$. Since $A$ is normal and since $xGx^{-1} = G$ it takes its values in $k[A]\cap G = A$. By
definition, the class of $\psi$ lies in the kernel of the homomorphism
\begin{equation}\label{qq}
H^1(G,A)\to H^1(G,k[A]^\cross)
\end{equation}
induced by the natural inclusion $A\to k[A]^\cross$. Since the field $k$ is algebraically closed, the algebra $k[A]$ is
isomorphic to the function algebra $k(\hat A)$. Moreover, this isomorphism preserves $G$-actions. Thus $k[A]^\cross$ is a
permutation $G$-module and $$H^1(G,k[A]^\cross) = \oplus_{\cO\subset \hat A} H^1(St_G(\cO),k^\cross) = \oplus_{\cO\subset
\hat A} \widehat{St_G(\cO)},$$ where the sum is taken over $G$-orbits in $\hat A$ and $St_G(\cO)$ is the stabilizer of an
orbit $\cO$. So the class of a cocycle $\psi$ is in the kernel of (\ref{qq}) iff $$\forall \chi\in \hat A,\quad
\chi(\psi(s)) = 1,\quad \forall s\in St_G(\chi),$$ which proves the proposition.
\end{proof}

\begin{exam}Quadratic class-preserving automorphisms.
\end{exam}
Let $V$ be a vector spaces over $\F=\F_2$ (the two element field). Let $b:V\otimes V\to \F$ be a non-degenerate
symmetric bilinear form on $V$ and $q:V\to \F$ an associated quadratic form: $$q(u+v) - q(u) - q(v) = b(u,v),\quad
u,v\in V.$$ Note that the existence of $q$ implies that $b$ is alternative: $b(v,v) = 0$ for any $v\in V$. Let $Q =
Aut(V,b)$ be the group of automorphisms of $b$ (a symplectic group) and $G=Q\ltimes V^*$ be a semi-direct product (an
affine symplectic group). Define an element in the group ring $k[V^*]$ of the dual space $V^*$ by $x = \sum_{v\in
V}(-1)^{q(v)}p_v$. Here the $p_v = \frac{1}{2^{dim(V)}}\sum_{l\in V^*}(-1)^{l(v)}l$ are minimal idempotents in $k[V^*]$
corresponding to elements of $V$. Then $$(x\otimes x)\Delta(x)^{-1} = \sum_{u,v\in V}(-1)^{b(u,v)}p_u\otimes p_v$$ is
$G$-invariant. The corresponding 1-cocycle $\psi:G\to V^*$ has the form
\begin{equation}\label{quadr}
\psi(g)(v) = q(v) - q(g(v)).
\end{equation}
Since $g$ preserves the linearisation of $q$, $\psi(g)$ is linear in $v$. Now it is straightforward to see that $[x,g]$
is given by the formula (\ref{quadr}): $$xg(x^{-1}) = \sum_{v\in V}(-1)^{q(v)}p_v\sum_{u\in V}(-1)^{-q(g(u))}p_u =
\sum_{v\in V}(-1)^{q(v)- q(g(v))}p_v = q(v) - q(g(v)).$$ Obviously the cocycle (\ref{quadr}) satisfies the condition
(\ref{absupcl}). To see that it is non-trivial we will follow the arguments of \cite{gr}. Recall that for a vector
$v\in V$ the map ({\em symplectic transvection}) $\tau_v(u) = u - b(v,u)v$ is an automorphism of $b$. Note that
symplectic transvections generate the group $Aut(V,b)$.

Write $V$ as $U\oplus U^*$ so that the form $b$ becomes $b((u,l),(v,m)) = l(v) - m(u)$. Define $q$ by $q(u,l) = l(u)$.
For a vector $(u,l)\in V$ the corresponding symplectic transvection has the form: $$\tau_{(u,l)}(v,m) = (v + (m(u) -
l(v))u, m + (m(u) - l(v))l).$$ Now the value of the 1-cocycle $\psi$ on the transvection is $$q(\tau_{(u,l)}(v,m)) -
q(v,m) = (m(u) + l(v))(l(u) + 1),$$ which coincides with $b((u,l),(v,m))(q(u,l) + 1)$. Thus we have the following
formula $$\psi(\tau_v)(u) = b(v,u)(q(v)+1),\quad v\in V.$$ The cocycle $\psi$ is a coboundary if there is a linear
function $l\in V^*$ such that $$l(\tau_v(u)) - l(u) = -b(v,u)l(v)$$ coincides with $\psi(\tau_v)(u) = b(v,u)(q(v)+1)$,
which gives a contradiction $l(v) = q(v) +1$ to the linearity of $l$.

\begin{exam}\label{ccpa} Cubic class-preserving automorphism of $3^5.M_{11}$.
\end{exam}
Let $V$ be a vector spaces over $\F=\F_3$ (the three element field). Let $\tau:V\otimes V\otimes V\to \F$ be a non-degenerate
symmetric tri-linear form on $V$ and $c:V\to \F$ an associated cubic form: $$c(u+v) - c(u) - c(v) = \tau(u,u,v) + \tau(u,v,v),\quad
u,v\in V.$$ Note that $c$ is a {\em depolarization} of $2\tau$:
$$c(u+v+w) - c(u+v) - c(u+w) - c(v+w) + c(u) + c(v) + c(w) = 2\tau(u,v,w)$$
and the form $\tau$ satisfies: $\tau(v,v,v) = 0$ for any $v\in V$. Let $Q =
Aut(V,\tau)$ be the group of automorphisms of the form $\tau$  and $G=Q\ltimes V^*$ be a semi-direct product. As in the previous example define an element in the group ring $k[V^*]$ of the dual space $V^*$ by $x = \sum_{v\in
V}\omega^{c(v)}p_v$. Here $\omega\in k$ is a primitive cubic root of unity and the $p_v = \frac{1}{3^{dim(V)}}\sum_{l\in V^*}\omega^{l(v)}l$ are minimal idempotents in $k[V^*]$
corresponding to elements of $V$. Since the group $Aut(V,\tau)$ preserves the function $\tau(u,u,v) + \tau(u,v,v)$, the element of $k[V^*]$: $$(x\otimes x)\Delta(x)^{-1} = \sum_{u,v\in V}(-1)^{\tau(u,u,v) + \tau(u,v,v)}p_u\otimes p_v$$ is
$G$-invariant. The corresponding 1-cocycle $\psi:G\to V^*$ again has the form
\begin{equation}\label{cubic}
\psi(g)(v) = c(v) - c(g(v)).
\end{equation}
Since $g$ preserves the linearisation of $q$ (which is $\tau(u,u,v) + \tau(u,v,v)$), $\psi(g)$ is linear in $v$.
In contrast with the previous example it is much more difficult to find a tri-linear form with a non-trivial 1-cocycle $\psi$.

Following \cite{wa}, define on $\F_{3^5}$ (the field of 243 elements), considered as a vectors space $V$ over $\F_3$, a symmetric tri-linear form $$\tau(x,y,z) = Tr(xyz^9 + xy^9z + x^9yz).$$ Here $Tr:\F_{3^5}\to \F_3$ is the trace of the field extension $\F_3\subset \F_{3^5}$. Let $\epsilon\in \F_{3^5}$ be a primitive root of unity of degree 11 such that $Tr(\epsilon) = -1$. It can be checked directly that the following linear operators $r,s,t$ on $V$ preserve $\tau$: $$s(v) = \epsilon v,\quad t(v) = v^3.$$ To define $r$ note that the powers $\epsilon, \epsilon^3,\epsilon^4,\epsilon^5,\epsilon^9$ span the space $V$. In this basis $r$ has the form $$r(\epsilon) = -\epsilon,\quad r(\epsilon^3) = -\epsilon^9,\quad r(\epsilon^4) = -\epsilon^5,$$ $$r(\epsilon^5) = -\epsilon^3,\quad r(\epsilon^9) = -\epsilon^4.$$ It was proved in \cite{wa} that $r,s,t$ generate the group $Aut(V,\tau)$ and that the group $Aut(V,\tau)$ is isomorphic to the Mathieu group $M_{11}$.  Now define $c$ by $c(x) = Tr(x^{11})$. It can be checked that $r^2,s,t$ stabilize $c$ (it was also proved in \cite{wa} that $r^2,s,t$ generate the subgroup of $M_{11}$ isomorphic to $PSL_2(11)$). Now if we assume that the 1-cocycle (\ref{cubic}) is a coboundary $$c(v) - c(g(v)) = l(v) - l(g(v)),\quad \forall v\in V$$ we should have a non-zero linear function $l\in V^*$ invariant under $s$, which is not possible. Indeed, writing $l$ as $l(v) = Tr(\lambda v)$ for some $\lambda\in \F_{3^5}$ we would have $Tr(\lambda\epsilon v) = Tr(\lambda v)$ for all $v\in V$, which implies $\lambda=0$.

\section{Invariant anti-symmetric twists and triangular structures on group algebras}

It follows from the proposition (\ref{separ}) that up to gauge transformations invariant anti-symmetric twists on the group algebra $k[G]$ of a group $G$ of odd order correspond to normal abelian subgroups $A\subset G$ with non-degenerate
alternative bi-multiplicative $G$-invariant forms $\alpha:A\times A\to k^\cross$: $$F_{(A,\beta)} =
\frac{1}{|A|}\sum_{a_1,a_2\in A}\beta(a_1,a_2)a_1\otimes a_2.$$

A bijective correspondence between non-degenerate bi-multiplicative forms $\beta:A\times A\to k^\cross$ on an abelian
group $A$ and non-degenerate bi-multiplicative forms $b:\hat A\times \hat A\to k^\cross$ on its group of characters
$\hat A$ can be defined explicitly by the rule: for any $x\in A$ there is a unique $\chi\in \hat A$ (and vice versa) such
that $$\beta(x,y) = \chi(y)\quad \forall y\in A\quad \Leftrightarrow\quad b(\chi,\psi) = \psi(x)\quad \forall
\psi\in\hat A.$$ This correspondence allows us to give a different presentation for anti-symmetric twists. Denote by
$p_\chi = \frac{1}{|A|}\sum_{x\in A}\chi^{-1}(x)x$ the minimal idempotent of the group algebra $k[A]$ corresponding to
the character $\chi\in\hat A$: $$yp_\chi = \chi(y)p_\chi,\quad \forall y\in A.$$
\begin{lem}
The anti-symmetric twist
$$F_{(A,\beta)} = \frac{1}{|A|}\sum_{a_1,a_2\in A}\beta(a_1,a_2)a_1\otimes a_2$$ can be written as
$$\sum_{\chi,\psi\in\hat A}b(\chi,\psi)p_\chi\otimes p_\psi.$$
\end{lem}
\begin{proof}
Indeed,
\begin{equation}\label{la}
\sum_{\chi,\psi\in\hat
A}b(\chi,\psi)p_\chi\otimes p_\psi = \frac{1}{|A|}\sum_{a_1,a_2\in A}\sum_{\chi,\psi\in\hat A}b(\chi,\psi)\chi(a_1)\psi(a_2))a_1\otimes a_2.
\end{equation}
Since $$b(\chi,\psi)\chi(a_1)\psi(a_2)) = \psi(x)b(x,a_1)\psi(a_2) = b(x,a_1)\psi(xa_2)$$ and
$$\frac{1}{|A|}\sum_{\psi\in\hat A}\psi(xa) = \delta_{x,a^{-1}}$$ the expression (\ref{la}) can be rewritten as
$$\frac{1}{|A|}\sum_{a_1,a_2\in A}\beta(a_2^{-1},a_1)a_1\otimes a_2.$$ It remains to notice that $\beta(a_2^{-1},a_1) =
\beta(a_1,a_2)$.
\end{proof}

For a subgroup $B$ of an abelian group $A$ with a non-degenerate alternative bi-multiplicative form $\beta:A\times
A\to k^\cross$ denote by $$B^\perp = \{ a\in A|\quad \beta(a,b) = 1, \forall b\in B\}$$ the {\em orthogonal complement}.
A subgroup $B$ is called {\em isotropic} if $B\subset B^\perp$, i.e. the restriction of $\alpha$ on $B$ is trivial. A subgroup
$B$ is {\em Lagrangian} if $B = B^\perp$. A Lagrangian subgroup $B\subset A$ fits into a short exact sequence $$B\to
A\to \hat B$$ with the last morphism being induced by the form $\beta$: $a\mapsto (x\in B\mapsto \beta(a,b)).$ Suppose
that there exists a multiplicative splitting $\hat B\to A$. Then we have a {\em Lagrangian decomposition} $$A\simeq B\oplus\hat
B$$ where elements of $A$ can be written as pairs $(x,\chi)$ where $x\in B$ and $\chi\in \hat B$. The form $\beta$ in this
presentation take the following shape $\beta((x,\chi),(y,\psi)) = \chi(y)\psi(x)^{-1}$.
\begin{lem}
An abelian group with an alternative bi-multiplicative non-degenerate form has a Lagrangian decomposition.
\end{lem}
\begin{proof}
We prove this by induction on the order of the group. Let $a\in A$ be an element with the property: for any $x\in A$ such
that $x^m=a$ there is $n$ so that $a^n = x$. Then the inclusion $\langle a\rangle\to A$ of the cyclic subgroup generated by $a$
splits as well as the surjection $A\to \widehat{\langle a\rangle}$ induced by the form $\beta$ on $A$. So we can write $A\simeq A'\oplus
\langle a\rangle\oplus \widehat{\langle a\rangle}$ where $A'\simeq \langle a\rangle^\perp/\langle a\rangle$ with the induced form. By the induction it follows that, $A'$ has a Lagrangian
decomposition, and hence so has $A$.
\end{proof}

With any Lagrangian decomposition $A\simeq B\oplus\hat B$ there are associated two more presentations for the twist
$F_{(A,\beta)}$, which, in a way, are mixtures of the previous two. Define $p_\chi = \frac{1}{|B|}\sum_{x\in
B}\chi(x)^{-1}x \in k[B]$.
\begin{lem}
The anti-symmetric twist
$$F_{(A,\beta)} = \frac{1}{|A|}\sum_{a_1,a_2\in A}\beta(a_1,a_2)a_1\otimes a_2$$ can be written as
$$F_{(A,\beta)} = \sum_{\chi,\psi\in \hat B}p_{\psi}\chi\otimes p_{\chi^{-1}}\psi.$$
\end{lem}
\begin{proof}
Indeed, $$\sum_{\chi,\psi\in \hat B}p_{\psi}\chi\otimes p_{\chi^{-1}}\psi = \frac{1}{|B|^2}\sum_{\chi,\psi\in \hat
B}\sum_{x,y\in B}\chi(y)\psi^{-1}(x)x\chi\otimes y\psi =$$ $$\frac{1}{|A|}\sum_{a_1,a_2\in A}\beta(a_1,a_2)a_1\otimes
a_2.$$
\end{proof}
Analogously (using the symmetry between $B$ and $\hat B$) we can write $$F_{(A,\beta)} = \sum_{x,y\in
B}xp_{y^{-1}}\chi\otimes yp_{x},$$ where $p_x = \frac{1}{|B|}\sum_{\chi\in \hat B}\chi(x)^{-1}\chi \in k[\hat B]$.

As the next example shows, not any Lagrangian subgroup fits into a Lagrangian decomposition.
\begin{exam}
\end{exam}
Let $A$ be $\Z/\Z_{p^2}\oplus \Z/\Z_{p^2}$ with the standard alternative form $\beta((1,0),(0,1)) = \varepsilon$ for
a primitive root $\varepsilon$ of degree $p^2$. The subgroup $B = pA$ is Lagrangian but the extension $B\to A\to \hat
B$ does not split.

Of course we can always split the short exact sequence $B\to A\to \hat B$ set-theoretically. Choosing a section
$$s:\hat B\to A,\quad \beta(s(\chi),x) = \chi(x),\ \forall x\in B, \chi\in \hat B$$ we can identify $A$ with
$B\times\hat B$ equipped with the product $$(x,\chi)(y,\psi) = (xy\Gamma(\chi,\psi),\chi\psi),$$ where $\Gamma:\hat B\times\hat
B\to B$ is a 2-cocycle defined by the splitting $s$: $$s(\chi\psi) = s(\chi)s(\psi)\Gamma(\chi,\psi).$$ The
bi-multiplicative form $\beta$ on $A$ transports to $B\times\hat B$: $$\beta((x,\chi),(y,\psi)) =
\chi(y)\psi(x)^{-1}\overline\beta(\chi,\psi),$$ where $\overline\beta(\chi,\psi) = \beta(s(\chi),s(\psi)).$

Repeating the calculations for the twist on Lagrangian decomposition, we will have the following form for the twist
$F_{(A,\beta)}$ on $B\times\hat B$:
\begin{equation}\label{twlag}
F_{(A,\beta)} = \sum_{\chi,\psi\in \hat B}\overline\beta(\chi,\psi)p_{\psi}\chi\otimes p_{\chi^{-1}}\psi.
\end{equation}

We finish this section with a well-known remark on triangular structures on group algebras. We include the proof of the proposition below because we will use its argument later on. 
\begin{prop}
The set $\Tr(k[G])$ of triangular structures on the group algebra is isomorphic to the set of pairs $(A,\alpha)$
consisting of a normal abelian subgroup $A\subset G$ and a non-degenerate skew-symmetric bi-multiplicative
$G$-invariant form $\alpha:A\times A\to k^\cross$.
\end{prop}
\begin{proof}
By the proposition (4.4.1) of \cite{da2} a support sub-bialgebra (minimal triangular sub-bialgebra) of a triangular structure
$R$ is a commutative and cocommutative normal sub-bialgebra $H_R$. Thus it is $k[A]$ for an abelian subgroup $A =
G(H_R)\subset G$, which must be normal since $H_R$ is normal. The isomorphism $b:H_R^*\to H_R$ induces (and is induced
by) the homomorphism of groups of group-like elements $G(b):G(H_R^*)\to G(H_R) = A$. Note that $G(H_R^*) =
Hom(A,k^\cross)$ is the group of characters of $A$. Thus the isomorphism $G(b):Hom(A,k^\cross)\to A$ induces (and is
induced by) a non-degenerate bi-multiplicative form $\alpha:A\times A\to k^\cross$. The $G$-invariance of $\alpha$ follows
from the $H$-invariance of $b$ while skew-symmetricity of $\alpha$ is equivalent to self-duality of $b$.
\end{proof}

\section{Group structure on anti-symmetric twists}

We start with the following abstract situation. Let $T$ be a group and $t:T\to T$ be an automorphism of order 2.
Define a left $T$-action on itself via
\begin{equation}\label{act}
x^g = t(g)^{-1}xg.
\end{equation}
In particular, the stabiliser of the identity element is $T^t = \{s\in T|\ t(s) = s\}$ the subgroup of $\tau$-invariant
elements. Note also that this action preserves the subset $$X = \{x\in T|\ t(x) = x^{-1}\}$$ and that $X$ is closed
under power maps $x\mapsto x^m$ for $m\in\Z$.
\begin{lem}\label{abstra}
Suppose that $(\ )^2$ is bijective on $X$. Then $$X\cap T^t = \{1\},\quad T = T^t X.$$ In particular, for any $x,y\in
X$, there is a unique factorisation: $$xy = s(x,y)(x\circ y),\quad s(x,y)\in T^t, x\circ y\in X$$ with $$(x\circ y)^2 =
yx^2y.$$ The operations $s,\circ$ satisfy $$s(x,y)s(x\circ y,z) = s(y,z)s(x,y\circ z),$$ $$(x\circ y)\circ z =
x^{s(y,z)}\circ(y\circ z),\quad (x\circ y)^s = x^s\circ y^s.$$ Here $x^s = s^{-1}xs$. Moreover if the operation $\circ$ is commutative, then $$[x,y] = s(x,y)s(y,x)^{-1}.$$ Here $[x,y] = xyx^{-1}y^{-1}$.
\end{lem}
\begin{proof}
First we notice that in the assumption of the lemma (the bijectivity of $(\ )^2$) $X$ coincides with the orbit
$\{t(g)^{-1}g, g\in T\}$ of the identity element. Indeed, solving $y^2 = x$ for $x\in X$ we will be able to write $x$
as $t(y)^{-1}y$ (since $y$ is also in $X$). The next thing to observe is that the intersection $X\cap T^t$ is trivial:
for $x\in X\cap T^t$, we have $x = t(x) = x^{-1}$ or $x^2 = 1$, which implies $x = 1$.

Now we write $T$ as the product $T^tX$. For $g\in T$, solve $y^2 = t(g)^{-1}g$ in $X$. Then for $s = gy^{-1}$, we have
$$t(s)s^{-1} = t(g)t(y)^{-1}yg^{-1} = t(g)y^2g^{-1} = 1.$$ So $s$ is in $T^t$. The decomposition $g = sy$ is unique
since $X\cap T^t = \{1\}$.

Since $T^t$ normalizes $X$ (the action of $T^t$ on $X$ via (\ref{act}) is the conjugation), we can write $xyz$ as an
element of $T^tX$ in two ways: $$(xy)z = s(x,y)(x\circ y)z = s(x,y)s(x\circ y,z)((x\circ y)\circ z),$$ $$x(yz) =
xs(y,z)(y\circ z) = s(y,z)x^{s(y,z)}(y\circ z) = s(y,z)s(x^{s(y,z)},(y\circ z))(x^{s(y,z)}\circ(y\circ z))$$ which
gives the equations of the lemma.

For $x\circ y = y\circ x$ we have $$[x,y] = (xy)(yx)^{-1} = s(x,y)(x\circ y)(y\circ x)^{-1}s(y,x)^{-1}  = s(x,y)s(y,x)^{-1}.$$
\end{proof}

Now let $T$ be the group $Aut^1_\Tw(k[G])$ of gauge classes of invariant twists with $t$ being the automorphism of
transposition of tensor functors of a twist. Then by the section \ref{clpr}, the stabiliser $T^t$ is the group
$Out_{cl}(G)$ of outer class-preserving automorphisms of $G$ and $X$ is (isomorphic to) the set of abelian normal
subgroups of $G$ equipped with $G$-invariant non-degenerate bimultiplicative form. If $G$ is odd then the order of any
element from $X$ is also odd so we are in the situation of the lemma \ref{abstra}. This defines a binary operation
$\circ$ on $X$: $$F_{(A_1,\alpha_1)}\circ F_{(A_2,\alpha_2)} = F_{(A,\alpha)},$$ where $$F_{(A,\alpha)}^2 =
F_{(A,\alpha^2)} = F_{(A_1,\alpha_1)}F_{(A_2,\alpha_2)}^2F_{(A_1,\alpha_1)}.$$ Note that the right hand side belongs to
the group algebra of the subgroup $A_1A_2\subset G$. Since the subgroups $A_i$ are normal abelian, their product
$A_1A_2$ is meta-abelian with the commutant $[A_1,A_2]\subset A_1\cap A_2\subset A_1A_2$.

First we describe the support group $A$. Let $$K = \{c\in A_1\cap A_2|\ \alpha_1(x,c) = \alpha_2(x,c)^{-1}\ \forall
x\in A_1\cap A_2\}$$ be the kernel of the bimultiplicative form on $A_1\cap A_2$, which is the product of the
restrictions of $\alpha_1$ and $\alpha_2$.
\begin{lem}\label{supp}
The formula $$\pi(a_1a_2,c) = \frac{\alpha_1(a_1,c)}{\alpha_2(a_2,c)},\ a_i\in A_i,\ c\in K$$ defines a homomorphism
$\pi:A_1A_2\to \hat K$ with an abelian kernel $A = ker(\pi)$. The formula $$\alpha(u_1u_2,v_1v_2) =
\alpha_1(u_1,v_1)\alpha_2(u_2,v_2)$$ defines a non-degenerate alternative bi-multiplicative form on $A$.
\end{lem}
\begin{proof}
The definition of $\pi$ does not depend on the factorisation $a_1a_2$. Then
\newline
for $d\in A_1\cap A_2$: $$\frac{\alpha_1(a_1d,c)}{\alpha_2(a_2,c)} =
\frac{\alpha_1(a_1,c)\alpha_1(d,c)}{\alpha_2(a_2,c)} =\frac{\alpha_1(a_1,c)}{\alpha_2(d,c)\alpha_2(a_2,c)} =
\frac{\alpha_1(a_1,c)}{\alpha_2(a_2d,c)}.$$ To see that $A$ is abelian note first that the commutant $[A_1,A_2]$
lies in $K$. To check that the commutant $[A,A]$ is trivial, it is enough to verify that $b_1([u_1u_2,v_1v_2],y) = 0$ for
any $u_1u_2,v_1v_2\in A$ and any $y\in A_1$. Writing $$[u_1u_2,v_1v_2] = [u_1,v_2][u_2,v_1] = [u_1,v_2][v_1,u_2]^{-1}$$
we need to verify that $b_1([u_1,v_2],y) = b_1([v_1,u_2],y)$. Indeed, by $G$-invariance of $b_i$ and the defining
relations for $u_1u_2,v_1v_2$ (together with $[A_1,A_2]\subset K$), we have the chain of equalities: $$b_1([u_1,v_2],y)
= b_1(u_1,[y,v_2])^{-1} = b_2(u_2,[y,v_2])^{-1} = $$ $$b_2([y,u_2],v_2) = b_1([y,u_2],v_1) = b_1(y,[v_1,u_2])^{-1} =
b_1([v_1,u_2],y).$$

We need to verify that the value of the form $\alpha$ does not depend on factorisations of its arguments. To see that, we give another presentation of the group $A$.
Denote by $A_1\bowtie A_2$ the group of pairs $(x_1,x_2)$, $x_i\in A_i$ with the product: $$(x_1,x_2)(y_1,y_2) =
(x_1y_1[x_2,y_1]^{\frac{1}{2}},[x_2,y_1]^{\frac{1}{2}}x_2y_2).$$ The map $$A_1\bowtie A_2\to G,\quad (x_1,x_2)\mapsto
x_1x_2$$ is a group homomorphism with the image $A_1A_2$. The group $A$ fits into a commutative diagram with exact rows
and columns:
\begin{equation}\label{exdiagm}
\xymatrix{ K \ar[r] \ar[d] & A_1\cap A_2 \ar[r] \ar[d] & \widehat{((A_1\cap A_2)/K)}  \ar[d] \\
ker(\Pi) \ar[r] \ar[d] &  A_1\bowtie A_2 \ar[r]^\Pi \ar[d] & \widehat{(A_1\cap A_2)}  \ar[d] \\
A \ar[r] & A_1A_2 \ar[r]^\pi & \hat{K} }
\end{equation}
The homomorphism $\Pi$ is given by the same formula as the map $\pi$. We can define a skew-symmetric pairing on $A_1\bowtie A_2$ using the formula $$\alpha((u_1,u_2),(v_1,v_2)) = \alpha_1(u_1,v_1)\alpha_2(u_2,v_2).$$ It follows from the non-degeneracy of $\alpha_i$ that the orthogonal to $ker(\Pi)$ in $A_1\bowtie A_2$ is the anti-diagonal image of $A_1\cap A_2$. Hence that the kernel of the restriction of the pairing from $A_1\bowtie A_2$ to $ker(\Pi)$ coincides with the anti-diagonal image of $K$. Thus the pairing restricts from $ker(\Pi)$ to a non-degenerate pairing on $ker(\Pi)/K=A$.
\end{proof}

\begin{prop}
The support subgroup and the bi-multiplicative form corresponding to the twist $F_{(A_1,\alpha_1)}\circ
F_{(A_2,\alpha_2)}$ are the group $A$ and the form $\alpha$ defined in the lemma \ref{supp}.
\end{prop}
\begin{proof}
First we prove that $A$ contains the support subgroup. Note that the supports of $F_{(A_1,\alpha_1)}\circ
F_{(A_2,\alpha_2)}$ and $R = F_{(A_1,\alpha_1)}F_{(A_2,\alpha_2)}^2F_{(A_1,\alpha_1)}$ coincide. To see that $R$
belongs to $k[A]\otimes k[A]$ it is enough to check that $(I\otimes\pi)(R) = 1$. Indeed, as an $R$-matrix, $R$ defines a
homomorphism of Hopf algebras $l\mapsto (l\otimes I)(R)$ whose image is the group algebra of the support subgroup of
$R$. The condition $(I\otimes\pi)(R) = 1$ will mean that the support subgroup is in the kernel of $\pi$. We check this
condition by showing that $(I\otimes\pi)(F_{(A_1,\alpha_1)}) = (I\otimes\pi)(F_{(A_2,\alpha_2)})^{-1}$ or, more
precisely, by showing that for any $c\in K$ the evaluation $(I\otimes\pi)(F_{(A_1,\alpha_1)})(I\otimes c)$ is the
inverse of $(I\otimes\pi)(F_{(A_2,\alpha_2)})(I\otimes c)$ (here we think of $k[\hat K]$ as the function algebra
$k[K]^*$). Since for $y\in A_1$ $\pi(y)(c) = \alpha_1(y,c)$ $$(I\otimes\pi)(F_{(A_1,\alpha_1)})(I\otimes c) =
\frac{1}{|A_1|}\sum_{x,y\in A_1}\alpha_1(x,y)x\otimes\pi(y)(c) = $$ $$\frac{1}{|A_1|}\sum_{x,y\in
A_1}\alpha_1(x,y)\alpha_1(y,c)x = \sum_{x\in A_1}(\frac{1}{|A_1|}\sum_{y\in A_1}\alpha_1(xc^{-1},y))x.$$ The inner sum
$\frac{1}{|A_1|}\sum_{y\in A_1}\alpha_1(xc^{-1},y)$ is the $\delta$-function $\delta_{x,c}$. Hence
$(I\otimes\pi)(F_{(A_1,\alpha_1)})(I\otimes c) = c$. Similarly, using that $\pi(y)(c) = \alpha_2(y,c)^{-1}$ for $y\in
A_2$, we get that $(I\otimes\pi)(F_{(A_2,\alpha_2)})(I\otimes c) = c^{-1}$.

Let $\chi_{u_1u_2}$ be the character on $A$ corresponding to an element $u_1u_2\in A$ via the form $\alpha$:
$$\chi_{u_1u_2}(x_1x_2) = \alpha_1(u_1,x_1)\alpha_2(u_2,x_2),\quad x_1x_2\in A.$$ To show that $\alpha^2$ is the form
corresponding to $R\in k[A]^{\otimes 2}$ we need to check that $(\chi_{u_1u_2}\otimes\chi_{v_1v_2})(R) =
\alpha(u_1u_2,v_1v_2)^{-2}$ for any $u_1u_2,v_1v_2\in A$. Write $R$ as $$\frac{1}{|A_1|^2|A_2|}\sum_{x_i,z_i\in A_1\\
y_i\in A_2}\alpha_1(x_1,x_2)\alpha_2(y_1,y_2)^2\alpha_1(z_1,z_2)x_1y_1z_1\otimes x_2y_2z_2.$$ Since $x_iy_iz_i = x_i\
^{y_i}{z_i}y_i$ (where $^{y_i}{z_i} = y_iz_iy_i^{-1}$), we have for $(\chi_{u_1u_2}\otimes\chi_{v_1v_2})(R)$
$$\frac{1}{|A_1|^2|A_2|}\sum_{x_i,z_i\in A_1\\ y_i\in A_2}\alpha_1(x_1,x_2)\alpha_2(y_1,y_2)^2\alpha_1(z_1,z_2)$$
$$\alpha_1(u_1,x_1\ ^{y_1}{z_1})\alpha_2(u_2,y_1)\alpha_1(v_1,x_2\ ^{y_2}{z_2})\alpha_2(v_2,y_2),$$ which by
multiplicativity and invariance of the forms $\alpha_i$ equals $$\frac{1}{|A_1|^2|A_2|}\sum_{x_i,z_i\in A_1\\ y_i\in
A_2}\alpha_1(x_1,x_2u_1^{-1})\alpha_1(z_1,^{y_1^{-1}}{u_1^{-1}}z_2)$$
$$\alpha_2(y_1,y_2)^2\alpha_2(u_2,y_1)\alpha_1(v_1,x_2\ ^{y_2}{z_2})\alpha_2(v_2,y_2).$$ This expression can be
simplified since $\frac{1}{|A_1|}\sum_{x_1\in A_1}\alpha_1(x_1,x_2u_1^{-1})$ is the $\delta$-function
$\delta_{x_2,u_1}$ and $\frac{1}{|A_1|}\sum_{x_1\in A_1}\alpha_1(z_1,^{y_1^{-1}}{u_1^{-1}}z_2)$ is the
$\delta$-function $\delta_{z_2,^{y_1^{-1}}{u_1}}$:
\begin{equation}\label{express}
\frac{1}{|A_2|}\sum_{y_i\in A_2}\alpha_2(y_1,y_2)^2\alpha_2(u_2,y_1)\alpha_1(v_1,u_1\
^{y_2y_1^{-1}}{u_1})\alpha_2(v_2,y_2).
\end{equation}
Now, writing $^{y_2y_1^{-1}}{u_1} = y_2y_1^{-1}\ ^{u_1}{(y_2y_1^{-1})^{-1}}u_1$ and using that $^{u_1^{-1}}{v_1} = v_1$
by commutativity of $A$, we have $$\alpha_1(v_1,u_1\ ^{y_2y_1^{-1}}{u_1}) =
\alpha_1(v_1,u_1)\alpha_1(v_1,y_2y_1^{-1})\alpha_1(v_1,^{u_1}{(y_2y_1^{-1})^{-1}})\alpha_1(v_1,u_1) = $$
$$\alpha_1(v_1,y_2y_1^{-1})\alpha_1(^{u_1^{-1}}{v_1},(y_2y_1^{-1})^{-1})\alpha_1(v_1,u_1)^2 = \alpha_1(v_1,u_1)^2.$$
This allows us to simplify the expression (\ref{express}) further: $$\frac{1}{|A_2|}\sum_{y_i\in
A_2}\alpha_2(y_1,y_2^2u_2^{-1})\alpha_1(v_1,u_1)^2\alpha_2(v_2,y_2).$$ Finally $\frac{1}{|A_2|}\sum_{y_1\in
A_2}\alpha_2(y_1,y_2^2u_2^{-1})$ is the $\delta$-function $\delta_{y_2,u_2}$, which leaves us with
\begin{equation}
\alpha_1(v_1,u_1)^2\alpha_1(v_2,u_2)^2 = \alpha(u_1u_2,v_1v_2)^{-2}.
\tag*{\qedhere}
\end{equation}
\end{proof}

Clearly the operation $\circ$ on anti-symmetric twists is commutative. Thus the commutator of any two anti-symmetric
twists $F_1 = F_{(A_1,b_1)}$ and $F_2 = F_{(A_2,b_2)}$ can be written as $$[F_1,F_2] = (u\otimes u)\Delta(u)^{-1}$$ for
some $u\in k[G]$. An explicit form of such $u$ is closely related to a certain invariant of the pair
$(A_1,b_1),(A_2,b_2)$, which we are going to describe now. The construction is very similar to the one for twists on universal enveloping algebras from \cite{da2}. Denote by $B = [A_1,A_2]$ the commutant of $A_1$ and $A_2$ in $G$. To a pair of characters
$(\chi_1,\chi_2)$ on $B$ we assign an element $[x_1,x_2]\in B$, where $x_i\in A_i$ are defined by $$b_i(x_i,x) =
\chi_i(x),\quad \forall x\in B.$$ The elements $x_i$ are defined up to subgroups $$B^\perp_{b_i} = \{y\in A_i|\
b_i(y,x) = 1\ \forall x\in B\}.$$ Note that these subgroups have the following commutation property:
$$[B^\perp_{b_1},A_2] = [A_1,B^\perp_{b_2}] = 1.$$ Indeed, for arbitrary $x\in B^\perp_{b_1}, y\in A_2$ and for any
$z\in A_1$ $$b_1([x,y],z) = b_1(x,[y,z]) = 1.$$ Hence $[x,y] = 1$. Similarly for $[A_1,B^\perp_{b_2}]$. Thus the
commutator $[x_1,x_2]$ depends only on $(\chi_1,\chi_2)$. Define a map
\begin{equation}\label{c}
c:\hat B\times\hat B\times\hat B\to k^\cross
\end{equation}
by $c(\chi_1,\chi_2,\chi_3) = \chi_3([x_1,x_2])$. This map is multiplicative in each variable. It's also $G$-invariant
by the construction. Moreover this map is symmetric. Indeed, let $y_i\in\ga_i$ be such that $$b_i(y_i,u) =
\chi_3(u),\quad \forall u\in B.$$ Then $$\chi_3([x_1,x_2]) = b_1(y_1,[x_1,x_2]) = b_1([x_2,y_1],x_1) =
\chi_1([y_1,x_2]),$$ which means that $c(\chi_1,\chi_2,\chi_3) = c(\chi_3,\chi_2,\chi_1)$. Similarly
$$\chi_3([x_1,x_2]) = b_2(y_2,[x_1,x_2]) = b_2([y_2,x_1],x_2) = \chi_2([x_1,y_2]),$$ which means that
$c(\chi_1,\chi_2,\chi_3) = c(\chi_1,\chi_3,\chi_2)$.

\begin{prop}
The commutator of two anti-symmetric twists $F_1 = F_{(A_1,b_1)}$ and $F_2 = F_{(A_2,b_2)}$ belongs to the group
algebra $k[[A_1,A_2]]$ of the commutator $B=[A_1,A_2]$ and has the form $$\sum_{\psi,\chi\in \hat B}
c(\chi,\chi,\psi)c(\chi,\psi,\psi)p_\psi\otimes p_\chi.$$ Here $c:\hat B\times\hat B\times\hat B\to k^\cross$ is the
map defined above.
\end{prop}
\begin{proof}
Note that $B$ is an isotropic subgroup in $(A_1,b_1)$ and $(A_2,b_2)$. Choose Lagrangian subgroups $B\subset B_i\subset
A_i$ and write (as in (\ref{twlag})) $$F_i = \sum_{\chi_i,\psi_i\in \hat
B_i}\overline\beta_i(\chi_i,\psi_i)p_{\psi_i}\chi_i\otimes p_{\chi_i^{-1}}\psi_i.$$ Note that the subgroups $B_i$ are
in the centre of $A_1A_2$. Thus we have the following formula for the commutator $[F_1,F_2] = F_1F_2t(F_1)t(F_2)$:
$$\sum\overline\beta_1(\chi_1,\psi_1)\overline\beta_2(\chi_2,\psi_2)\overline\beta_1({\chi'}_1,{\psi'}_1)
\overline\beta_2({\chi'}_2,{\psi'}_2)$$
$$p_{\psi_1}p_{\psi_2}p_{{\chi'}_1^{-1}}p_{{\chi'}_2^{-1}}\chi_1\chi_2{\psi'}_1{\psi'}_2 \otimes
p_{\chi_1^{-1}}p_{\chi_2^{-1}}p_{{\psi'}_1}p_{\psi_2}\psi_1\psi_2{\chi'}_1{\chi'}_2,$$ which immediately gives
${\chi'}_1=\psi_1^{-1},\ {\chi'}_2=\psi_2^{-1},\ {\psi'}_1=\chi_1^{-1},\ {\psi'}_2=\chi_2^{-1}$. Since the groups $B_i$ do not
have 2-torsion $\overline\beta_i(\chi_i,\psi_i)\overline\beta_i(\psi_i^{-1},\chi_i^{-1}) =1$ and the commutator takes
the form: $$\sum_{\chi_1,\psi_i}p_{\psi_1}p_{\psi_2}\chi_1\chi_2\chi_1^{-1}\chi_2^{-1}\otimes
p_{\chi_1^{-1}}p_{\chi_2^{-1}}\psi_1\psi_2\psi_1^{-1}\psi_2^{-1}.$$ Obviously, $p_{\psi_1}p_{\psi_2} = 0$ if the
restrictions of $\psi_1$ and $\psi_2$ to $B$ do not coincide. Note that $\sum_{\psi_i|_B = \psi}p_{\psi_1}p_{\psi_2} =
p_\psi$ and that $\psi([\chi_1,\chi_2]) = c(\chi,\chi,\psi),$ $\chi([\psi_1,\psi_2]) = c(\chi,\psi\psi)$ so the
commutator is equal to
\begin{equation}
\sum_{\psi,\chi\in \hat B}c(\chi,\chi,\psi)c(\chi,\psi,\psi)p_\psi\otimes p_\chi.
\tag*{\qedhere}
\end{equation}
\end{proof}

Now to find an element $u = \sum_{\chi\in \hat B}u(\chi)p_\chi\in k[B]$ such that $$[F_1,F_2] = (u\otimes
u)\Delta(u)^{-1}$$ it is enough to solve the equation $$u(\chi\psi) =
c(\chi,\chi,\psi)c(\chi,\psi,\psi)u(\chi)u(\psi).$$ Since $c$ is symmetric we can always find a solution. It is much
more subtle to check that this solution will be $G$-invariant, which will guarantee that $u$ is in the centre of $k[G]$. However if $3$ is coprime to $|G|$ then we can always do that by solving $u(\chi)^3 = c(\chi,\chi,\chi)$. Thus we have the
following result.
\begin{theo}
For a finite group $G$ whose order is not divisible by $2$ and $3$ the group of isomorphism classes $Aut_\Tw(k[G])$ of
twisted automorphisms of the group ring $k[G]$ is isomorphic to the semi-direct product of $Out(G)\ltimes
Aut^{anti}_\Tw(k[G])$, where $Aut^{anti}_\Tw(k[G])$ coincides with the abelian group of anti-symmetric twists; that is, the group of pairs
$(A,b)$ consisting of a normal subgroup $A\subset G$ and an alternative bi-multiplicative form $b:A\times A\to
k^\cross$.
\end{theo}

\begin{exam}Heisenberg group.
\end{exam}
For an elementary abelian $p$-group $E$ with a symplectic form $b:E\otimes E\to \F_p$ with values in the prime field
$\F_p$, let $H = H(E,b)$ be the corresponding Heisenberg group:

Fix a primitive $p$-th root of unity $\epsilon$ in $k$. Any anti-symmetric invariant twist in $k[H]^{\otimes 2}$ has a
form: $$F_x = \frac{1}{p^2}\sum_{i,j,s,t=0}^{p-1}\epsilon^{it-js}x^ic^s\otimes x^jc^t =
\sum_{i,j=0}^{p-1}x^ip_{-j}\otimes x^jp_i$$ for some non-central $x\in E$. Here $p_i =
\frac{1}{p}\sum_{t=0}^{p-1}\epsilon^{-it}c^t$ are central idempotents.

Since $F_x^{-1} = F_{x^{-1}}$, the commutator of two invariant twists can be expressed as follows: $$[F_x,F_y] = \sum
x^{i_1}p_{j_1}y^{k_1}p_{l_1}x^{-i_2}p_{j_2}y^{-k_2}p_{l_2}\otimes
x^{j_1}p_{-i_1}y^{l_1}p_{-k_1}x^{j_2}p_{i_2}y^{l_2}p_{k_2} =$$ $$\sum_{i,j=0}^{p-1}[x^i,y^i]p_j\otimes [x^j,y^j]p_i =
\sum_{i,j=0}^{p-1}c^{i^2b(x,y)}p_j\otimes c^{j^2b(x,y)}p_i = \sum_{i,j=0}^{p-1}\epsilon^{(i^2j+j^2i)b(x,y)}p_j\otimes
p_i.$$ Note that $i^2j+j^2i = \frac{1}{3}((i+j)^3 - i^3 - j^3)$. Thus for $\eta\in k$ such that $\eta^3 = \epsilon$ the
element $u=u(x,y) = \sum_{k=0}^{p-1}\eta^{-kb(x,y)}p_k$ satisfies $$(u\otimes u)\Delta(u)^{-1} = [F_x,F_y].$$

\begin{exam}Non-commuting invariant anti-symmetric twists.
\end{exam}
Here, following the Lie algebra case, we construct a group with two abelian normal subgroups equipped with invariant bi-multiplicative forms with given invariant (\ref{c}). Let $C$ be an abelian group and $c:C\oplus C\oplus C\to k^\cross$ be a symmetric tri-multiplicative map. Define on $M(C,c) = C\oplus C\oplus \hat C$ the structure of a meta-abelian group: $$(x_1,y_1,\chi_1)(x_2,y_2,\chi_2) = (x_1+x_2,y_1+y_2,\chi_1\chi_2c(x_1,y_2,-)).$$ The subgroups $A_1 = C\oplus \{0\}\oplus\hat C,\ A_2 = \{0\}\oplus C\oplus\hat C$ are normal abelian. The standard alternative bi-multiplicative forms $b_i$ on $A_2$ are invariant. The commutant of the invariant twists $F_i=F_{(A_i,b_i)}$ has the form $$[F_1,F_2] = (u\otimes u)\Delta(u)^{-1}$$ for $$u = \sum_{x,y\in C}
c(x,x,y)c(x,y,y)p_x\otimes p_y,$$ where $p_x = |C|^{-1}\sum_{\chi\in\hat C}\chi(x)^{-1}(0,0,\chi)\in k[\hat C]\subset k[M(C,c)]$.

Let $Q$ be a group acting on $C$ and preserving $c$.  Then $Q$ acts by automorphisms on $M(C,c)$ and this action preserves subgroups $A_i$ with forms $b_i$. Thus the twists $F_i=F_{(A_i,b_i)}$ remain invariant for the semi-direct product $G=M(C,c)\rtimes Q$. Applying this construction to the tri-multiplicative map $c = e^{\frac{2\pi i\tau}{3}}$ corresponding to the tri-linear form $\tau$ from the example (\ref{ccpa}) with $Q=Aut(\tau)$, we get an example of two invariant anti-symmetric twists whose commutant corresponds to a non-trivial class-preserving automorphism.

\end{document}